\documentclass[12pt]{amsart}

\usepackage{fullpage}
\usepackage{amsfonts,amscd}
\usepackage{amssymb}
\usepackage{url}
\usepackage{textcomp}

\usepackage[english]{babel}

\allowdisplaybreaks

\theoremstyle{plain}
\newtheorem{theorem}                {Theorem}      [section]
\newtheorem{proposition}  [theorem]  {Proposition}
\newtheorem{corollary}    [theorem]  {Corollary}
\newtheorem{lemma}        [theorem]  {Lemma}

\theoremstyle{definition}

\newtheorem{remark}       [theorem]  {Remark}
\newtheorem{definition}   [theorem]  {Definition}

\newtheorem{problem}      [theorem]  {Problem}

\numberwithin{equation}{section}

\def \R{{\mathbb R}}

\def \s {{\mathbb S}}
\def \h {{\mathbb H}}

\def \Hy{{\mathbb H}}
\def \C{{\mathbb C}}

\usepackage{color}

\def \link {~}
\def \1 {\`}

\DeclareMathOperator{\sol}{Sol}

\numberwithin{equation}{section}

\title[Polyharmonic helices]{Polyharmonic helices}

\author{S.~Montaldo}
\address{Universit\`a degli Studi di Cagliari\\
Dipartimento di Matematica e Informatica\\
Via Ospedale 72\\
09124 Cagliari, Italia}
\email{montaldo@unica.it}

\author{A.~Ratto}
\address{Universit\`a degli Studi di Cagliari\\
Dipartimento di Matematica e Informatica\\
Via Ospedale 72\\
09124 Cagliari, Italia}
\email{rattoa@unica.it}

\usepackage{hyperref}
\begin{document}
\begin{abstract}
The main aim of this paper is to investigate the existence of Frenet helices which are polyharmonic of order $r$, shortly, $r$-harmonic. We shall obtain existence, non-existence and classification results. More specifically, we obtain a complete classification of proper $r$-harmonic helices into the $3$-dimensional solvable Lie group Sol$_3$. Next, we investigate the existence of proper $r$-harmonic helices into Bianchi-Cartan-Vranceanu spaces and, in this context, we find new examples. Finally, we shall establish some non-existence results both for Frenet curves and Frenet helices of order $n \geq 4$ when the ambient space is the Euclidean sphere $\s^m$. 
\end{abstract}



\maketitle
\section{Introduction}

\textit{Harmonic maps} are the critical points of the {\em energy functional}
\begin{equation}\label{energia}
E(\varphi)=\frac{1}{2}\int_{M_1}\,|d\varphi|^2\,dV_{g_1} \, ,
\end{equation}
where $\varphi:M_1\to M_2$ is a smooth map between two Riemannian
manifolds $(M_1,g_1)$ and $(M_2,g_2)$. We refer to \cite{MR703510, MR1363513} for an introduction to this important topic. 

The study of \textit{polyharmonic maps of order} $r$, shortly \textit{$r$-harmonic maps}, was first proposed by Eells and Lemaire in \cite{MR703510}. These maps are defined as the critical points of the functionals 
\[
E_{r}(\varphi)=\frac{1}{2}\int_{M_1}\,|(d+d^*)^r\varphi|^2\,dV_{g_1} 
\]
which represent a version of order $r\geq2$ of the classical energy \eqref{energia}. We refer to \cite{MR4106647} for a detailed discussion on the definition of the functionals $E_{r}(\varphi)$.

In this context, the most widely studied instance is $r=2$. This is the case of the so-called \textit{bienergy} functional and its critical points are known as \textit{biharmonic maps}. At present, a very ample literature on biharmonic maps is available (see \cite{MR4265170} and references therein).

More recently, several authors worked intensively on the case $ r \geq 3$. For instance, we refer to \cite{MR4106647, MR4444191, MR4632927, MR4598081,MR2869168, MR3403738, MR3371364, MOR-Israel, MR3711937, MR3790367} for background and results.

In this paper we shall focus on the study of \textit{$r$-harmonic curves} in the Euclidean sphere $\s^m$, in the $3$-dimensional solvable Lie group Sol$_3$ and into the so-called Bianchi-Cartan-Vranceanu spaces. Thus, in order to simplify the exposition, from now on we shall assume that the ambient space $M$ is orientable and that an orientation on $M$ has been fixed.

In order to avoid trivialities arising from re-parametrization it is convenient to restrict the attention to curves parametrized by the arc length and we shall adopt this assumption throughout the whole paper. 

Now, let $\gamma:I \subset \R \to M$ denote a smooth curve parametrized by the arc length $s$ and let $T$ denote its unit tangent vector. Then $\gamma$ is $r$-harmonic if and only if its $r$-tension field $\tau_r(\gamma)$ vanishes, i.e.,(see \cite{MR4542687,MR2869168,MR4308322})
\begin{equation}\label{r-harmonicity-curves}
\tau_r(\gamma)=\nabla^{2r-1}_T T+ \sum_{\ell=0}^{r-2}(-1)^\ell R\left (\nabla^{2r-3-\ell}_T T ,\nabla^{\ell}_T T\right )T=0 \,,
\end{equation}
where $\nabla^{0}_T T=T,\,\nabla^{k}_T T=\nabla_T \left ( \nabla^{k-1}_T T\right ) $ and $R$ is the Riemannian curvature operator of $M$.

We point out that any geodesic is trivially $r$-harmonic for all $r \geq 1$. Therefore, we say that $\gamma$ is a \textit{proper} $r$-harmonic curve if it is $r$-harmonic and \textit{not} a geodesic.

A general study of \eqref{r-harmonicity-curves} is a difficult task. Therefore, in this paper we restrict our attention to a geometrically rich and significant family of curves. Indeed,
\begin{definition}[\textbf{$n$-Frenet curves}]\label{def-curves} Let $\gamma:I \to M$ be a smooth curve parametrized by arc length and assume $2 \leq n \leq m=\dim M$. Then we say that $\gamma$ is a \textit{Frenet curve of order} $n$, shortly, an \textit{$n$-Frenet curve}, if it admits a Frenet $n$-field $\{F_1,\ldots,F_n \}$ along $\gamma$ such that
\begin{equation}\label{Frenet-field-general}
\left \{
\begin{array}{ccl}
F_1&=&T \\
\nabla_T F_1&=& k_1 \,F_2 \\
\nabla_T F_i&=&-k_{i-1} F_{i-1}+k_i F_{i+1}\,, \quad i=2,\ldots, n-1\,,\\
\nabla_T F_n&=&-k_{n-1} F_{n-1} 
\end{array} \right .
\end{equation}
where $k_1(s), \ldots, k_{n-1}(s)$ are positive functions, called the \textit{curvatures} of the curve. In the special case that $n=m\geq 2$  we choose $F_n$ in such a way that $\{F_1,\ldots,F_n \}$ is a positively oriented orthonormal base of $T_{\gamma(s)}M$. In this case we allow that $k_{n-1}(s)$ can assume any real value. Finally, we say that an $n$-Frenet curve $\gamma$ is \textit{full} if $k_{n-1}(s)$ does not vanish at any point.
\end{definition}
\begin{definition}[\textbf{$n$-Frenet helices}]\label{def-helices} Let $\gamma$ be a Frenet curve of order $n$. We say that $\gamma$ is \textit{a Frenet helix of order $n$}, shortly, an \textit{$n$-Frenet helix}, if its curvatures $k_1(s), \ldots, k_{n-1}(s)$ are constant functions which will be denoted $\kappa_1,\ldots,\kappa_{n-1}$.
\end{definition}

\begin{remark} In order to make easier the comparison of our results with those which are available in the literature, in the case of $2$-Frenet helices we shall write $\{T,N\}$ instead of $\{F_1,F_2\}$ and $\kappa$ for $\kappa_1$. 

As for $3$-Frenet helices, we write $\{T,N,B\}$ for $\{F_1,F_2,F_3\}$, and $\kappa,\tau$ for $\kappa_1,\kappa_2$. 
%

Moreover, we point out that we adopt the following notation and sign convention for the Riemannian curvature tensor field $R$ of the ambient space $M$:
\begin{equation}\notag
 {R}(X,Y)Z={\nabla}_{X}{\nabla}_{Y}Z
-{\nabla}_{Y}{\nabla}_{X}Z-{\nabla}_{[X,Y]}Z\,, \quad  {R}(X,Y,Z,W)=\langle {R}(X,Y)W,Z \rangle \,.
\end{equation}
\end{remark}
\vspace{2mm}

The Chen-Maeta conjecture states that any $r$-harmonic submanifold of $\R^n$ is minimal. As for curves, Maeta (see \cite{MR2869168}) proved that the conjecture is true.

The study of triharmonic helices in space forms was initiated by Maeta (see \cite{MR3007953}) and, more recently, Branding carried out a complete study of $r$-harmonic Frenet helices of order $3$ into $3$-dimensional space forms (see \cite{MR4542687}). More precisely, he described explicit examples in $\s^3$ and proved non-existence when the curvature of the ambient space form is non-positive. 

This last result supports the idea that, in general, it is difficult to find examples of proper $r$-harmonic submanifolds in an ambient space whose Riemannian curvature tensor field is non-positive. Despite this, the non-existence of proper biharmonic submanifolds into a non-positively curved ambient space, known as generalized Chen's conjecture, was first proved to be false in \cite{MR2975260}. Moreover, other examples were also constructed in a conformally flat, non-positively curved ambient space (see \cite{MR3095129}).

An investigation of triharmonic $2$-Frenet helices in surfaces and triharmonic $3$-Frenet helices in $3$-dimensional homogeneous spaces with a $4$-dimensional group of isometries was carried out in \cite{MR4308322}. More recently, triharmonic $3$-Frenet helices in the $3$-dimensional solvable Lie group Sol$_3$ were fully classified in \cite{Suceava}. 
\vspace{2mm}

In this paper we shall make some progress concerning $r$-harmonic $n$-Frenet helices, $n=2,3$, into the $3$-dimensional solvable Lie group Sol$_3$ and into Bianchi-Cartan-Vranceanu spaces. Our results in this context will be described and proved in Sections\link\ref{Sec-Sol} and \ref{Sec-BCV} respectively. 
\vspace{2mm}

As for general results, our first contribution is
\begin{theorem}\label{Th-2-Frenet-general}
Let $\gamma:I \to M$ be a Frenet helix of order $2$ and $r \geq 2$. Then
\[
\tau_r(\gamma)=-\kappa^{2r-3}\Big[\kappa^2\,N+(r-1) R(T,N)T \Big ]\,.
\] 
\end{theorem}
\begin{corollary}\label{Cor-2-Frenet-general} Let $\gamma:I \to M$ be a Frenet helix of order $2$ and $r \geq 2$. If
\[
\langle R(T,N)T,W \rangle =0 \quad \forall\, W \in \big({\rm Span}\{T,N\}\big)^\perp \,,
\]
then $\gamma$ is proper $r$-harmonic if and only if
\begin{equation}\label{eq-r-harmonicity-2-helix-general}
\kappa^2=(r-1) R(N,T,N,T) \,.
\end{equation}
\end{corollary}
\vspace{2mm}

Next, we examine the case of $3$-Frenet helices. To this purpose, let us introduce the following quantities:
\begin{equation}\label{def-A1A2A3}
A_1= \kappa^2 (r-1)+\tau ^2\,,\quad A_2= -   \kappa \,\tau \,(r - 2)\,,\quad A_3= -\left(\kappa^2+\tau ^2\right)^2\,.
\end{equation}
Then we can state our result:
\begin{theorem}\label{Th-3-Frenet-general}
Let $\gamma:I \to M$ be a Frenet helix of order $3$ and $r \geq 2$. Then
\[
\tau_r(\gamma)=c\,\Big \{A_3 N+A_1 R(N,T)T+A_2R(N,B)T \Big \}\,,
\]
where $c=(-1)^r \kappa \left(\kappa^2+\tau ^2\right)^{r-3}$ and $A_i,\,i=1,2,3,$ are defined in \eqref{def-A1A2A3}.
\end{theorem}
\begin{corollary}\label{Cor-3-Frenet-general}
Let $\gamma:I \to M$ be a Frenet helix of order $3$ and $r \geq 2$. Assume that
\[
\langle R(N,T)T,W \rangle =0 \,, \,\langle R(N,B)T,W \rangle =0\quad \forall\, W \in \big ({\rm Span}\{T,N,B\}\big)^\perp \,.
\]
Then $\gamma$ is proper $r$-harmonic if and only if the following two equations hold:
\begin{equation}\label{eq-r-harmonicity-3-helix-general}
\begin{array}{ll}
{\rm (i)}&A_1 R(N,T,N,T)+A_2 R(N,B,N,T)+A_3=0 \\
&\\
{\rm (ii)}&A_1 R(N,T,B,T)+A_2 R(N,B,B,T)=0\,.
\end{array} 
\end{equation}
\end{corollary}

\begin{remark}\quad
\begin{itemize}
\item[(i)] Putting together Theorem\link\ref{Th-2-Frenet-general} and Corollary\link\ref{Cor-2-Frenet-general} we conclude that condition \eqref{eq-r-harmonicity-2-helix-general} is a necessary condition for a $2$-Frenet helix to be $r$-harmonic. This condition is also sufficient if one proves that $\tau_r(\gamma)$ belongs to ${\rm Span}  \{T,N\}$. For instance, this surely happens if the ambient space is a surface or any $m$-dimensional space form, $m \geq 2$. 
Similarly, we deduce from Theorem\link\ref{Th-3-Frenet-general} and Corollary\link\ref{Cor-3-Frenet-general} that the two equations in \eqref{eq-r-harmonicity-3-helix-general} represent a necessary condition for the $r$-harmonicity of a $3$-Frenet helix, and this condition is also sufficient if one proves that $\tau_r(\gamma)$ belongs to ${\rm Span} \{T,N,B\}$. For instance, this is automatic if the ambient space has dimension $3$ or it is an $m$-dimensional space form, $m \geq 3$.

\item[(ii)]  Theorems\link\ref{Th-2-Frenet-general} and \ref{Th-3-Frenet-general}  imply that the component of $\tau_r(\gamma)$ along $T$ vanishes. Therefore, these helices are always $r$-conservative.

\item[(iii)]  In the special case that the ambient space $M$ is a $3$-dimensional space form of curvature $K$ it is easy to check that equation \eqref{eq-r-harmonicity-3-helix-general}(ii) always holds, while \eqref{eq-r-harmonicity-3-helix-general}(i) is equivalent to
\[
\left(\kappa^2+\tau ^2\right)^2=K \left((r-1)\kappa^2+\tau ^2\right)\,.
\]
Then we have recovered Theorem 1.1 of \cite{MR4542687} as a special case of our Theorem\link\ref{Th-3-Frenet-general}--Corollary\link\ref{Cor-3-Frenet-general}.
\end{itemize}
\end{remark}

Our next result suggests that helices and, more generally, curves of high order cannot be $r$-harmonic for small values of $r$. More precisely, we have:
\begin{theorem}\label{Th-non-existence-r-curves}
Let $\gamma$ be a full $n$-Frenet curve in $\s^m$, $m\geq n \geq 2$. If $ n \geq 2r$, then $\gamma$ cannot be $r$-harmonic.
\end{theorem}

In the final section of this paper we shall prove that a stronger non-existence statement can be obtained if we assume that the curve is a helix (see Section\link\ref{n-FrenethelicesinSn}).

\begin{remark}
For the sake of completeness, here we point out that the choice of the most suitable definition of a general helix in a Riemannian geometric context is a delicate matter and may depend on the specific situation. Definition\link\ref{def-helices} above seemed to us the most appropriate for our approach to the study of $r$-harmonic curves. By contrast, the following geometrically interesting definition has been used in a significant paper of Barrios (see \cite{MR1363411}) along the lines of previous work by Langer and Singer (see \cite{MR772124}).
\begin{definition}\label{Def-Barros} \cite{MR1363411} A curve $\gamma(s)$ into a $3$-dimensional space form will be called a \textit{general helix} if there exists a Killing vector field $\mathcal{V }(s)$ with constant length along $\gamma$ and such that the angle between $\mathcal{V }$ and the unit tangent $T$ is a non-zero constant along $\gamma$. We will say that $\mathcal{V }$ is an axis of the general helix.
\end{definition}
In particular, this definition does not imply that the curvatures of the helix are constant, and this would boost the analytical difficulties in our context. The interested reader can find some further comments in this direction in the last section of our recent paper \cite{Suceava}.
\end{remark}

Our paper is organized as follows. 

The results stated in this introduction will be proved in Section\link\ref{Sec-proofs}. The classification and existence results concerning $r$-harmonic $n$-Frenet helices, $n=2,3$, in the $3$-dimensional solvable Lie group Sol$_3$ and in BCV-spaces will be stated and proved in Sections\link\ref{Sec-Sol} and \ref{Sec-BCV} respectively.

Next, in Section\link\ref{n-FrenethelicesinSn} we shall study full, $r$-harmonic $n$-Frenet helices in $\s^m$ and prove some non-existence results. In order to overcome some technical difficulties in this context we shall make use of some  properties of Groebner bases. 

Finally, at the end of the paper we have added a related appendix where we discuss some conditions which imply that a proper triharmonic curve has constant geodesic curvature.

\section{Proof of the results stated in the introduction}\label{Sec-proofs}
\begin{proof}[Proof of Theorem\link\ref{Th-2-Frenet-general}] First, an elementary argument using
\[
\left \{
\begin{array}{ccl}
\nabla_T T&=& \kappa \,N \\
\nabla_T N&=&-\kappa \,T
\end{array} \right .
\]
shows that, for all $\ell \geq 0$,
\begin{equation}\label{formula-nablaelleTT-2-Frenet}
\left \{
\begin{array}{lll}
\nabla_T^{2\ell} T&=& (-1)^{\ell}\kappa^{2\ell} \,T \\
&&\\
\nabla_T^{2\ell+1} T&=& (-1)^{\ell}\kappa^{2\ell+1} \,N\,.
\end{array} \right .
\end{equation}
Next, let us assume that $r=2s$. Then, in order to be able to use \eqref{formula-nablaelleTT-2-Frenet}, it is convenient to rewrite the $r$ tension field  $\tau_r(\gamma)$ given in \eqref{r-harmonicity-curves} splitting odd and even terms as follows:
\begin{equation}\label{split-taur}
\tau_{2s}(\gamma)=\nabla_T^{4s-1} T + \sum_{j=0}^{s-1}R\left(\nabla_T^{2(2s-2-j)+1} T,\nabla_T^{2j} T \right )T-\sum_{j=0}^{s-2}R\left(\nabla_T^{2(2s-2-j)} T,\nabla_T^{2j+1} T \right )T \, .
\end{equation}
Next, inserting \eqref{formula-nablaelleTT-2-Frenet} into this expression and simplifying we deduce that
\[
\tau_{2s}(\gamma)=-\kappa^{2r-1}N-\Big [\sum_{j=0}^{s-1} \kappa^{2r-3}+\sum_{j=0}^{s-2}\kappa^{2r-3}\Big] R(T,N)T =-\kappa^{2r-3}\Big[\kappa^2\,N+(r-1) R(T,N)T \Big ]
\]

and so the proof is completed in this case. When $r=2s+1$ the proof is analogous and so we omit the details.
\end{proof}
\begin{proof}[Proof of Corollary\link\ref{Cor-2-Frenet-general}]
It suffices to observe that $\langle \tau_r(\gamma),T\rangle =0$ and compute $\langle \tau_r(\gamma),N\rangle $.
\end{proof}
\begin{proof}[Proof of Theorem\link\ref{Th-3-Frenet-general}] First, we recall a lemma for $3$-Frenet helices which was proved by Branding (see Lemma\link2.1 of \cite{MR4542687}):
\begin{lemma}\label{lemma-Branding-nablaelle} Let
\[
\begin{array}{lll}
h(\ell)=(-1)^\ell \kappa^2 \big (\kappa^2 + \tau^2 \big )^{\ell-1} &{\rm if} &\ell \geq 1\,, h(0)=1;\\
g(\ell)=(-1)^{\ell-1} \kappa\,\tau \big (\kappa^2 + \tau^2 \big )^{\ell-1}&{\rm if} &\ell \geq 1\,, g(0)=0;\\ f(\ell)= (-1)^{\ell}  \kappa \big (\kappa^2 + \tau^2 \big )^{\ell}&{\rm if} &\ell \geq 0 \,.
\end{array}
\]
Then, for all $\ell \geq 0$, we have:
\begin{equation}\label{formula-nablaelleTT-3-Frenet}
\begin{cases}
\nabla_T^{2\ell} T= h(\ell) \,T+ g(\ell) \,B \\
\nabla_T^{2\ell+1} T= f(\ell)\,N\,.
\end{cases}
\end{equation}
\end{lemma}
Next, let $r=2s$. We use again the split \eqref{split-taur}. More precisely, inserting \eqref{formula-nablaelleTT-3-Frenet} we obtain:
\begin{eqnarray*}
\sum_{j=0}^{s-1}R\left(\nabla_T^{2(2s-2-j)+1} T,\nabla_T^{2j} T \right )T&=&\Big (\sum_{j=0}^{s-1}f(2s-2-j)h(j) \Big ) R(N,T)T \\
&&+\Big (\sum_{j=0}^{s-1}f(2s-2-j)g(j) \Big )R(N,B)T\\
&=&\Big (f(2s-2)+\sum_{j=1}^{s-1}f(2s-2-j)h(j) \Big ) R(N,T)T\\
&&+\Big (\sum_{j=1}^{s-1}f(2s-2-j)g(j) \Big )R(N,B)T\,.
\end{eqnarray*}
Similarly,
\begin{eqnarray*}
-\sum_{j=0}^{s-2}R\left(\nabla_T^{2(2s-2-j)} T,\nabla_T^{2j+1} T \right )T&=&\Big (-\sum_{j=1}^{s-1}f(j-1)h(2s-1-j) \Big ) R(T,N)T \\
&&-\Big (\sum_{j=1}^{s-1}f(j-1)g(2s-1-j) \Big )R(B,N)T \,.
\end{eqnarray*}
Next, using these last two computations and simplifying we obtain:
\begin{eqnarray*}
\tau_r(\gamma)&=&f(2s-1)N\\
&&+\Big(f(2s-2)+\sum_{j=1}^{s-1}\big[ f(2s-2-j)h(j)+f(j-1)h(2s-1-j)\big] \Big )R(N,T)T\\
&&+\Big (\sum_{j=1}^{s-1}\big [f(2s-2-j)g(j)+ f(j-1)g(2s-1-j)\big ]\Big ) R(N,B)T\,.
\end{eqnarray*}

Next, a simple, explicit computation of the sums in $\tau_r(\gamma)$ gives
\[
\tau_r(\gamma)=c\,\Big \{A_3 N+A_1 R(N,T)T+A_2R(N,B)T \Big \}\,,
\]
where $c=(-1)^r \kappa \left(\kappa^2+\tau ^2\right)^{r-3}$ and $A_i,\,i=1,2,3$ are defined in \eqref{def-A1A2A3}. So the proof is completed in the case that $r=2s$. When $r=2s+1$ the proof is analogous and so we omit the details.
\end{proof}
\begin{proof}[Proof of Corollary\link\ref{Cor-3-Frenet-general}]
It suffices to observe that $\langle \tau_r(\gamma),T\rangle =0$ and compute $\langle \tau_r(\gamma),N\rangle $ and $\langle \tau_r(\gamma),B\rangle $. Then we easily obtain equations \eqref{eq-r-harmonicity-3-helix-general}(i) and \eqref{eq-r-harmonicity-3-helix-general}(ii) respectively. More precisely, under the assumptions of this corollary, we have
\begin{eqnarray*}
\tau_r(\gamma)&=&c\Big\{\Big(A_1 R(N,T,N,T)+A_2 R(N,B,N,T)+A_3\Big ) N\\
&&\,\,\quad +\Big (A_1 R(N,T,B,T)+A_2 R(N,B,B,T)\Big ) B\Big \}
\end{eqnarray*}
as required to end the proof.
\end{proof}
\begin{proof}[Proof of Theorem \ref{Th-non-existence-r-curves}] First, we recall that the curvature tensor field $R$ on $\s^m$ is given by:
\begin{equation}\label{Curv-tensor-sfera}
R(X,Y)Z=-\langle X,Z \rangle Y+ \langle Y,Z \rangle X \,.
\end{equation}
Next, using the structural equations \eqref{Frenet-field-general} of the Frenet field it is elementary to show by induction that, for any $0 \leq \ell \leq n-1$, we can write
\begin{equation}\label{Nablaelle}
\nabla^\ell_T T=\sum_{j=1}^{\ell} a_j(s) F_j + \Pi_{j=1}^{\ell}k_j(s) F_{\ell+1}\,,
\end{equation}
where the explicit expression of the coefficients $a_j(s)$ is not relevant for our purposes.

Now, taking into account \eqref{r-harmonicity-curves}, \eqref{Curv-tensor-sfera} and \eqref{Nablaelle}, it is immediate to deduce that, whenever $2 r \leq n$, we can write
\[
\tau_r(\gamma)= \sum_{j=1}^{2r-1} b_j(s) F_j+\Pi_{j=1}^{2r-1}k_j(s) F_{2r}
\]
where, again, the explicit expression of the coefficients $b_j(s)$ is not relevant for our purposes. Now, inspection of the component of $\tau_r(\gamma)$ along $F_{2r}$ readily ends the proof.
\end{proof}

\section{$r$-harmonic Frenet helices in the $3$-dimensional solvable Lie group $\sol_3$}\label{Sec-Sol}
In this section we shall obtain a complete classification of $r$-harmonic Frenet helices in the $3$-dimensional solvable Lie group $\sol_3$. 

Let $\sol_3$ $=(\mathbb{R}^{3},g_{\sol})$ denote the $3$-dimensional Riemannian manifold given by $\R^3$ endowed with the metric tensor $g_{\sol}=e^{2z}{\rm d}x^{2}+e^{-2z}{\rm
d}y^{2}+{\rm d}z^{2}$ with respect to the standard Cartesian coordinates
$(x,y,z)$. 

The space $\sol_3$ is a $3$-dimensional solvable Lie group and the metric $g_{\sol}$ is left-invariant with respect to the group operation
\[
(x,y,z).(x',y',z')=\left (e^{z} x'+x,e^{-z}y'+y,z'+z \right )\,.
\] 
A global orthonormal frame field on $\sol_3$ is
\begin{equation}\notag
E_{1}=e^{-z}\frac{\partial}{\partial x},\;
E_{2}=e^{z}\frac{\partial}{\partial y},
\;E_{3}=\frac{\partial}{\partial z}.
\end{equation}
With respect to this orthonormal frame field, the Lie brackets and the
Levi-Civita connection can be easily computed. We have:
\begin{equation}\notag
[E_{1},E_{2}]=0, \;[E_{2},E_{3}]=-E_{2}, \;[E_{1},E_{3}]=E_{1},
\end{equation}
\begin{equation}\label{Derivate-Ei-Ej}
\begin{array}{lll}
\nabla_{E_{1}}E_{1}=-E_{3},&\nabla_{E_{1}}E_{2}=0, &\nabla_{E_{1}}E_{3}=E_{1}\\
\nabla_{E_{2}}E_{1}=0,& \nabla_{E_{2}}E_{2}=E_{3},&\nabla_{E_{2}}E_{3}=-E_{2}\\
\nabla_{E_{3}}E_{1}=0,&\nabla_{E_{3}}E_{2}=0,&\nabla_{E_{3}}E_{3}=0.\\
\end{array}
\end{equation}
A further computation gives
\begin{equation}\label{curv-1}
\begin{array}{lll}
R(E_1,E_2)E_1=-E_{2},& R(E_1,E_3)E_1=E_{3},
&R(E_1,E_2)E_2=E_{1},\\R(E_2,E_3)E_2=E_{3},&
R(E_1,E_3)E_3=-E_{1},&R(E_2,E_3)E_3=-E_{2}
\end{array}
\end{equation}
and from this we deduce
\begin{equation}\label{curv-2}
\begin{array}{lll}
 R_{1212}=1,\\
R_{1313}=-1,\\
R_{2323}=-1.
\end{array}
\end{equation}
Let $\gamma(s)=(x(s),y(s),z(s))$ be a $3$-Frenet curve in $\sol_3$ parametrized by the arc length $s$. Then
\begin{equation}\label{T}
T={x'}e^z E_1 +{y'}e^{-z} E_2+{z'} E_3=\left ( T_1,T_2,T_3 \right )\,.
\end{equation}

Now, let $V=V_1(s)E_1+V_2(s)E_2+V_3(s)E_3=\left ( V_1(s),V_2(s),V_3(s) \right )$ be any vector field along $\gamma$. Then a computation using \eqref{Derivate-Ei-Ej} and \eqref{T} shows that
\begin{equation}\label{nablaTV}
\nabla_T V=\left ( {V_1}'+T_1 V_3,{V_2}'-T_2 V_3,{V_3}'-T_1 V_1+T_2V_2\right )\,.
\end{equation}
Our main result in this context is:
\begin{theorem}\label{Th-main-Existence-Sol} Let $r\geq 3$. Then there exists a proper $r$-harmonic $3$-Frenet helix $\gamma$ in $\sol_3$ if and only if 
\[
\kappa=\frac{1}{4} \sqrt{\frac{3 r^2-10 r+7}{(r-2)^2}} \quad {\rm and} 
\quad \tau=\pm\frac{r-1}{4 (r-2)} \,.
\]
Moreover, up to isometries of $\sol_3$ and orientation of $\gamma$, the parametrization of such a proper $r$-harmonic helix is
\begin{equation}\label{parametrizations}
\gamma(s)=\left (-\frac{\sqrt{(1-c_1^2)}}{c_1\sqrt{2}} e^{-c_1s},\, \frac{\sqrt{(1-c_1^2)}}{c_1\sqrt{2 }} e^{c_1s}\,,\,c_1\,s \right) \,,
\end{equation}
where 
\begin{equation}\label{c1-sol}
c_1=\frac{1}{2} \sqrt{\frac{3 r-7}{r-2}} \,.
\end{equation}
\end{theorem}

Since the ambient space is $3$-dimensional, the analysis of Theorem\link\ref{Th-main-Existence-Sol} enables us to deduce:
\begin{corollary}\label{Cor-2-Frenet-helices-Sol} Let $r\geq 3$. Then there exists no proper $r$-harmonic $2$-Frenet helix in $\sol_3$.
\end{corollary}
\begin{remark}In \cite{Caddeo-altri-Mediterr-2006} the authors proved that there is no proper biharmonic curve in $\sol_3$. The special case $r=3$ of Theorem\link\ref{Th-main-Existence-Sol} was proved in \cite{Suceava}.
\end{remark}
\begin{proof}[Proof of Theorem \ref{Th-main-Existence-Sol}]
We assume that $B=T \wedge N$ so that $\{T,N,B\}$ is a positively oriented basis. Then,  the Frenet equations are:
\[
\left \{
\begin{array}{lll}
\nabla_T T&=&\kappa \,N \\
\nabla_T N&=&-\kappa \,T+ \tau \,B  \\
\nabla_T B&=&- \tau  \,N  \,
\end{array}
\right .
\]
and, according to Definition~\ref{def-curves}, in this case we allow that $\tau$ may be negative. Now, using \eqref{curv-1} and \eqref{curv-2}, we compute
\begin{equation}\label{curv-3}
\begin{array}{lll}
R(T,N,T,N)&=&-1+2 B_3^2 \\
R(T,N,T,B)&=&-2 N_3 B_3 \\
R(B,N,N,T)&=&2 T_3 B_3 \\
R(B,N,B,T)&=&-2 T_3 N_3 \,.
\end{array}
\end{equation}

Next, we need the following technical Lemma:
\begin{lemma}\label{Th-N3=0} Let $\gamma$ be a proper $r$-harmonic $3$-Frenet helix in $\sol_3$, $ r \geq 3$.
Then 
\begin{itemize}
\item [\rm(i)] $N_3\equiv 0$ along $\gamma$.
\item [\rm(ii)] Both $T_3$ and $B_3$ are constant along $\gamma$.
\end{itemize}
\end{lemma}
\begin{proof}[Proof of Lemma\link\ref{Th-N3=0}] This lemma was proved in \cite{Suceava} in the case that $r=3$ (see Lemma 4.1 of \cite{Suceava}). The proof in the case that $r \geq 3$ is essentially the same and so we omit further details.
\end{proof}

Now, a routine computation shows that for any $3$-Frenet helix $\gamma$ in $\sol_3$ we have:
\begin{equation}\label{T3N3B3}
\begin{array}{lll}
{\rm (i)}&T_3=&z' \\
&&\\
{\rm (ii)}&N_3=&-\displaystyle{\frac{1}{\kappa} \big ( e^{2z}x'^2- e^{-2z}y'^2-z'' \big )}\\
&&\\
{\rm (iii)}&B_3=&\displaystyle{\frac{1}{\kappa} \big ( -y'x''+x' (-4y'z'+y'') \big )} 
\end{array}
\end{equation}
Since we work up to isometries and orientation of $\gamma$, $T_3$ constant implies that we can assume
\begin{equation}\label{z1}
z(s)=c_1 s 
\end{equation} 
for some constant $c_1$ with $0 \leq c_1\leq 1$. Moreover, as a consequence of Lemma\link\ref{Th-N3=0}, we can assume that $N_3=0$ along $\gamma$. 

Then integration of the condition $N_3=0$ using \eqref{T3N3B3}(ii) and \eqref{z1} yields
\begin{equation}\label{yprimoxprimo}
y'=c_2 e^{2c_1s}x'
\end{equation}
where $c_2>0$ is an integration constant.
Next, using \eqref{T}, \eqref{z1} and imposing $\langle T,T \rangle=1$ we deduce
\[
c_1^2+2 c_2 e^{2c_1s}x'^2=1 
\]
along $\gamma$. Integration of this equality tells us that either
\begin{equation}\label{x1}
x(s)=-\frac{\sqrt{(1-c_1^2)}}{c_1\sqrt{2 c_2}} e^{-c_1s}+c_x 
\end{equation}
or 
\begin{equation}\label{xtilde}
x(s)=\frac{\sqrt{(1-c_1^2)}}{c_1\sqrt{2 c_2}} e^{-c_1s}+c_x 
\end{equation}
where $c_x \in \R $ is an integration constant. Arguing as in \cite{Suceava}, we see that it is not restrictive to assume that we are in the case \eqref{x1} and that $c_2=1$, $c_x=0$. 

Now, direct substitution into \eqref{yprimoxprimo} and integration enables us to deduce that
\[
y(s)=\frac{\sqrt{(1-c_1^2)}}{c_1\sqrt{2 }} e^{c_1s}+c_y  \,,
\]
and, again, we can assume that $c_y=0$.
Now, a straightforward computation using these explicit expressions for $x(s),y(s),z(s)$ and \eqref{nablaTV} shows that:
\[
\kappa^2=c_1^2-c_1^4 \,, \quad \tau=1-c_1^2 \,.
\]
Finally, using these values for $\kappa$ and $\tau$ into \eqref{eq-r-harmonicity-3-helix-general}, together with \eqref{curv-3} and $B_3=\pm \sqrt{1-c_1^2}$, it is straightforward to check that the only acceptable value for $c_1$ is \eqref{c1-sol}. This value for $c_1$ gives the stated values of $\kappa$ and $\tau$ (positive) and so the proof of Theorem\link\ref{Th-main-Existence-Sol} is completed in the case \eqref{x1}. The case \eqref{xtilde} is analogous and gives the negative value of $\tau=c_1^2-1$.
\end{proof}

\section{$r$-harmonic Frenet helices in BCV-spaces}\label{Sec-BCV}

It is well-known (see, e.g., \cite{Belkhelfa-altri-book-2000}, \cite{Caddeo-altri-Mediterr-2006}, \cite{manzano-torralbo}, \cite{Daniel-Comm-Math-Helv-2007}) that $3$-dimensional homogeneous spaces with group of isometries of dimension $4$ admit, as a {\it local} canonical model, the so called  Bianchi-Cartan-Vranceanu  spaces (shortly, BCV-spaces)
\begin{equation}\label{CV}
M^3_{m,\ell}=\left(\bar{M}\times \R,g=\frac{dx^2+dy^2}{[1+m(x^2+y^2)]^2}+\Big[dz+\frac{\ell}{2}\,\frac{y
dx-x dy}{1+m(x^2+y^2)}\Big ]^2\right)\,,
\end{equation}
where $\bar{M}=\{(x,y)\in\R^2\colon 1+m(x^2+y^2)>0\}$.

The space $M^3_{m,\ell}$ is the total space of the following Riemannian submersion over a simply connected surface $M^2(4m)$ of constant curvature $4m$, see \cite{Daniel-Comm-Math-Helv-2007}: 
\begin{eqnarray}\label{RSM}
\pi:
M^3_{m,\ell}\longrightarrow  M^2(4m)=\left(\bar{M},h=\frac{dx^2+dy^2}{[1+m(x^2+y^2)]^2}\right),\;\;\;
\pi(x,y,z)=(x, y).
\end{eqnarray}
We point out that while in $\s^n(\rho)$ the letter $\rho$ indicates the radius, in $M^2(4m),\,\h^2(4m)$ the real number within the brackets represents the sectional curvature.

These BCV-spaces are also a local model for the Thurston eight $3$-dimensional geometries with the exception of the
hyperbolic space $\Hy^3(m)$ and $\sol_3$. More precisely, we have the following cases: 
\begin{enumerate}
\item when $\ell=m=0$, then  $M^3_{m,\ell}$  is $ \R ^3$;
\item when $\ell^2=4m>0$, then $M^3_{m,\ell}$  is a local model of $ \s^3(1/\sqrt{m})$;
\item when $\ell=0$ and $m<0$, then  $M^3_{m,\ell}$  is the product space ${\mathbb H}^2(4m)\times \R$;
\item when $\ell=0$ and $m>0$, then  $M^3_{m,\ell}$  is the product space $\s^2(\frac{1}{2\sqrt{m}})\setminus \{\infty\}\times \R$;
\item when $\ell\neq 0$ and $m=0$, then $M^3_{m,\ell}$ is the Heisenberg space ${\rm Nil_3}$;
\item when $\ell\neq 0$ and $m< 0$, then $M^3_{m,\ell}$ is  $\widetilde{\rm SL}(2,\R)$;
\item  when $\ell\neq 0,\, m>0$ and  $\ell^2\neq 4m$, then $M^3_{m,\ell}$  is a local model of ${\rm SU}(2)$ with a Berger metric.
\end{enumerate}

The aim of this section is to investigate the existence of $r$-harmonic $3$-Frenet helices in BCV-spaces. 

We shall mainly focus on the case $r\geq 3$ because the biharmonic case was thoroughly investigated in \cite{Caddeo-altri-Mediterr-2006,MR2129448}.

We begin by introducing some basic preliminary facts. First, it is easy to check that the vector fields
\begin{equation}\notag
E_{1}=F\frac{\partial}{\partial
x}-\frac{\ell y}{2}\frac{\partial}{\partial z},\quad E_{2}=F
\frac{\partial}{\partial y}+\frac{\ell x}{2}\frac{\partial}{\partial
z},\quad E_{3}=\frac{\partial}{\partial z},
\end{equation}
where $F=1+m(x^2+y^2)$, form a global orthonormal
frame field. Moreover,
\begin{equation}\label{CNil}
\begin{array}{ll}
\nabla_{E_{1}}E_{1}=2myE_{2}\,,& \nabla_{E_{2}}E_{2}=2mxE_{1}\,,\vspace{2mm} \\
\nabla_{E_{1}}E_{2}=-2myE_{1}+\dfrac{\ell}{2}E_{3}\,,&
\nabla_{E_{2}}E_{1}=-2mxE_{2}-\dfrac{\ell}{2}E_{3}\,,\vspace{2mm} \\
\nabla_{E_{3}}E_{1}=\nabla_{E_{1}}E_{3}=-\dfrac{\ell}{2}E_{2}\,,& \nabla_{E_{3}}E_{2}=\nabla_{E_{2}}E_{3}=\dfrac{\ell}{2}E_{1}\,,\vspace{2mm} \\
{\rm all \;\; others\;\;} \nabla_{E_i}E_j=0,\;i,j=1, 2, 3.&
\end{array}
\end{equation}
The Riemannian curvature operator $R$ of $M^3_{m,\ell}$ can be described as follows (see \cite{Daniel-Comm-Math-Helv-2007}, where the opposite sign convention for the curvature tensor is adopted):
\begin{eqnarray}\label{tensore-curvatura-general-expression}
 R(X,Y)Z&=& \left ( 4m - \frac{3 \ell^2}{4}\right )\big( - \langle X,Z \rangle Y  + \langle Y,Z \rangle X \big )-(4 m - \ell^2)\nonumber \\
&& \big( \langle Y,E_3 \rangle \langle Z,E_3 \rangle X+
 \langle Y,Z \rangle \langle X,E_3 \rangle E_3-\langle X,E_3 \rangle \langle Z,E_3 \rangle Y-
 \langle X,Z \rangle \langle Y,E_3 \rangle E_3\big) \,.
\end{eqnarray}
Then a simple computation gives the possible non-zero values of the sectional curvatures:
\begin{equation}\label{BCV1}
\begin{split}
 R_{1212}&=\langle R(E_{1},E_{2})E_{2},E_{1} \rangle=4m-\frac{3\ell ^2}{4},\\
R_{1313}&=\langle R(E_{1},E_{3})E_{3},E_{1}\rangle =\frac{\ell ^2}{4},\\
R_{2323}&=\langle R(E_{2},E_{3})E_{3},E_{2} \rangle=\frac{\ell ^2}{4}\,.
\end{split}
\end{equation}

Now, let $\gamma:I \to M_{m,\ell}^3$ be a curve parametrized by the arc length $s$ in a BCV-space. Then, with respect to the orthonormal frame field $\left \{ E_1,E_2,E_3 \right\}$,
\[
T=\left (T_1,T_2,T_3 \right )=\left ( \frac{x'}{F},\,\frac{y'}{F},\,\frac{2 z' \,F+\ell ( y x'- x y')}{2 \,F}\right ) \,.
\]
Moreover, a computation based on \eqref{CNil} shows that, if $V=\left(V_1,V_2,V_3 \right)$, we have:
\begin{eqnarray*}
\nabla_T V&=&\left(V'_1,V'_2,V'_3\right )+\Big ( -2 m y T_1V_2+2mxT_2V_2+\frac{\ell}{2}\left (T_2V_3+T_3V_2 \right ),
\\
&&2my T_1V_1-2 mx T_2V_1-\frac{\ell}{2}\left (T_1V_3+T_3V_1 \right ),\frac{\ell}{2}\left (T_1V_2-T_2V_1 \right ) \Big )\,.
\end{eqnarray*}
As in the proof of Theorem\link\ref{Th-main-Existence-Sol}, we can assume $B=T \wedge N$. Now, using \eqref{BCV1}, we compute:
\begin{equation}\label{curv-3-BCV}
\begin{array}{lll}
R(T,N,T,N)&=&\frac{\ell ^2}{4}+(4 m - \ell ^2)B_3^2 \\
R(T,N,T,B)&=&-B_3 N_3 \left(4 m-\ell ^2\right) \\
R(B,N,N,T)&=& T_3 B_3 \left(4 m-\ell ^2\right) \\ 
R(B,N,B,T)&=&- N_3 T_3 \left(4 m-\ell ^2\right)\,.
\end{array}
\end{equation}
Next, a straightforward computation using Lemma\link\ref{lemma-Branding-nablaelle}, Corollary\link\ref{Cor-3-Frenet-general} and \eqref{curv-3-BCV} leads us to our first general result:
\begin{proposition}\label{Prop-taur-BCV}
Let $\gamma$ be a $3$-Frenet helix in a BCV-space $M^3_{m,\ell}$. Then
\begin{equation}\label{TAUr-BCV-elica}
\tau_r(\gamma)= \tau_{r,N} N + \tau_{r,B} B \,,
\end{equation}
where
\begin{equation}\label{TAUrN-BCV}
\tau_{r,N} =\frac{1}{4} \left(4 B_3^2 \left(4 m-\ell ^2\right)+\ell ^2\right) \left( (r-1)\kappa^2+\tau ^2\right)+(r-2)T_3 B_3 \kappa  \tau  \left(4 m-\ell ^2\right)-\left(\kappa^2+\tau ^2\right)^2
\end{equation}
and
\begin{equation}\label{TAUrB-BCV}
\tau_{r,B} =-N_3 \left(4 m-\ell^2\right) \left(B_3 \left(\kappa^2 (r-1)+\tau ^2\right)+(r-2) T_3\kappa \tau  \right)\,.
\end{equation}
\end{proposition}
Now, the main objective of this section becomes the description of the best applications of Proposition\link\ref{Prop-taur-BCV}.

Since the case of space forms has been thoroughly discussed in \cite{MR4542687}, from now on we shall restrict our attention to the case that $4m \neq \ell^2$. First, we establish the following important technical lemma:
\begin{lemma}\label{Th-N3=0-BCV} Let $\gamma$ be a proper $r$-harmonic $3$-Frenet helix in a BCV-space $M_{m,\ell}^3$, $ r \geq 2$, $4m \neq \ell^2$.
Then 
\begin{itemize}
\item [\rm(i)] $N_3\equiv 0$ along $\gamma$.
\item [\rm(ii)] Both $T_3$ and $B_3$ are constant along $\gamma$ and $B_3\neq 0$.
\end{itemize}
\end{lemma}
\begin{proof}[Proof of Lemma\link\ref{Th-N3=0-BCV}] Although the statement is similar to Lemma\link\ref{Th-N3=0}, the proof is substantially different and so we include it.

(i) First, we observe that
\begin{eqnarray}\label{BprimeTprime}\nonumber
B_3'&=&\nabla_T B_3=\nabla_T\langle B, E_3 \rangle= \langle \nabla_T B, E_3 \rangle+\langle B,\nabla_T E_3 \rangle\\
&=&-\tau \,N_3+\langle B,\frac{\ell}{2}\,T_2 E_1-\frac{\ell}{2}\,T_1 E_2 \rangle=-\tau \,N_3+\frac{\ell}{2}\, N_3 \\\nonumber
T_3'&=&\nabla_T T_3=\nabla_T\langle T, E_3 \rangle= \langle \nabla_T T, E_3 \rangle+\langle T,\nabla_T E_3 \rangle\\\nonumber
&=&\kappa \,N_3+\langle T,\frac{\ell}{2}\,T_2 E_1-\frac{\ell}{2}\,T_1 E_2 \rangle=\kappa \,N_3 \,.
\end{eqnarray}

We argue by contradiction and assume that $N_3 \neq 0$ on some open interval. Since $\tau_{r,B}=0$, we deduce that
\begin{equation}\label{B3-BCV}
B_3=-\frac{\kappa (r-2) \tau  T_3}{\left(\kappa^2 (r-1)+\tau ^2\right)}\,.
\end{equation}
Differentiating with respect to $s$ and using \eqref{BprimeTprime} we deduce that
\[
\left (- \tau+\frac{\ell}{2}\right ) N_3=-\frac{\kappa^2 (r-2) \tau  }{\left(\kappa^2 (r-1)+\tau ^2\right)}\,N_3\,.
\]
Now, since $N_3 \neq 0$, we easily deduce that
\begin{equation}\label{ell-BCV}
\ell=\frac{2 \tau  \left(\kappa^2+\tau ^2\right)}{\kappa^2 (r-1)+\tau ^2}\,.
\end{equation}
Finally, inserting \eqref{B3-BCV} and \eqref{ell-BCV} into \eqref{TAUrN-BCV} and simplifying we find
\[
\tau_{r,N}=-\frac{\kappa^2 (r-1) \left(k^2+\tau ^2\right)^2}{\kappa^2 (r-1)+\tau ^2}\,.
\]
Therefore, since this quantity cannot vanish, we conclude that there are no solutions with $N_3\neq 0$, as required to end part (i).

(ii) From \eqref{BprimeTprime} we have $T_3'=\kappa \,N_3$ and so the conclusion follows from part (i).
Finally, if $N_3=B_3=0$ then $T=E_3$ and $\gamma$ would be a geodesic.\end{proof}

Next, we prepare the ground to give explicit parametrizations of $r$-harmonic $3$-Frenet helices. 

Let us assume that $\gamma(s)=(x(s),y(s),z(s))$ is a non-geodesic arc-length parametrized curve immersed in a BCV-space $M_{m,\ell}^3$ with $4m \neq \ell^2$ and satisfying $N_3=0$. Then, following the computations of \cite[\S 5.2]{Caddeo-altri-Mediterr-2006}, we have that there exists a constant $\alpha_o\in(0,\pi)$ and a unique (up to an additive constant $2k\pi$) smooth function $\beta$ such that we can assume that the Frenet frame field along $\gamma$ is given by
\begin{equation}\label{frenetN3=0}
\begin{cases}
T(s)=\sin\alpha_o\cos\beta(s)E_1+\sin\alpha_o\sin\beta(s)E_2+\cos\alpha_o E_3 \,,\\
N(s)=-\sin\beta(s)E_1+\cos\beta(s)E_2 \,,\\
B(s)=T(s)\wedge N(s)=-\cos\alpha_o\cos\beta(s)E_1-\cos\alpha_o\sin\beta(s)E_2+\sin\alpha_o E_3 \,.
\end{cases}
\end{equation}
Moreover,  the curvature and the torsion of $\gamma(s)$ are given by
\begin{eqnarray}
\kappa(s)&=&\zeta \sin\alpha_o\,,\label{ks}\\
\tau(s)&=&\zeta\cos\alpha_o+\frac{\ell}{2}\,,\label{ts}
\end{eqnarray}
where 
\begin{equation}\label{beta(s)}
\zeta=|\beta'(s)+2m\sin\alpha_o\left[y\cos\beta(s)-x\sin\beta(s)\right]-\ell\cos\alpha_o|>0 \,.
\end{equation}
We point out that in \cite{Caddeo-altri-Mediterr-2006} the sign convention for $\tau$ is different because in \cite{Caddeo-altri-Mediterr-2006} the authors assumed
$$
\nabla_T B=\tau \,N \,.
$$

If we also require that $\gamma(s)$ is a $3$-Frenet helix, then $\zeta$ is constant. Furthermore, substituting the expressions \eqref{ks} and \eqref{ts} in \eqref{TAUrN-BCV} we conclude that any such \textcolor{red}{a} $3$-Frenet helix is proper $r$-harmonic if and only if the constant $\zeta$ is a positive root of the following polynomial of degree four:
\begin{eqnarray}\label{Pzeta}
&&P_4(\zeta)=8 \zeta^4+16 \ell \cos\alpha_o \zeta^3\\
&&+\left[\cos (2 \alpha_0 ) \left(16 m (r-1)-3 \ell^2 (r-2)\right)-16 m (r-1)+\ell^2 (3 r+4)\right]\zeta^2\nonumber\\
    &&+2 \ell \cos\alpha_o \left[\ell^2 
   -2r (4 m - \ell^2) \sin^2\alpha_o\right]\zeta-2 \ell^2\left(4m-\ell^2\right)\sin^2\alpha_o\,.
\nonumber\end{eqnarray}

\subsection{The case $m=0$}

When $m=0$ and $\ell\neq 0$ the BCV-space $M_{0,\ell}^3$ represents the Heisenberg space ${\rm Nil_3}$. More precisely, following the terminology of \cite{Daniel-Comm-Math-Helv-2007}, ${\rm Nil_3}$ is $\R^3$ endowed with the metric \eqref{CV}  with $m=0$. We point out that the standard metric on ${\rm Nil_3}$ corresponds to the choice $\ell=1$.
In this case, the polynomial \eqref{Pzeta} becomes
\begin{eqnarray}\label{polynomial-m=0}
P_4(\zeta)\lvert_{m=0}&=&8 \zeta^4+16 \ell \cos\alpha_o \zeta^3-\ell^2\left[ (3 (r-2) \cos (2 \alpha_o)-3 r-4)\right]\zeta^2\\\nonumber
    &&+2 \ell^3 \cos \alpha_o \left[1+2
   r \sin^2\alpha_o\right]\zeta+2 \ell^4\sin^2\alpha_o
\end{eqnarray}
while \eqref{beta(s)} reduces to 
\begin{equation}\label{eq:betaprimo-m=0}
\zeta=\left |\beta'(s)-\ell \cos\alpha_o \right |\,.
\end{equation}
If $\beta(s)=\beta_o$ is constant, then $\cos\alpha_o=\pm \zeta /\ell$. Replacing  $\cos\alpha_o=- \zeta /\ell$ in \eqref{polynomial-m=0} gives 
\[
-2\,\left(\zeta ^2-\ell^2\right) \left(\ell^2+\zeta ^2 (r-2)\right)\,.
\]
In this case, the only positive root is $\zeta=|\ell|$ which is not acceptable as we would have $\cos\alpha_o=\pm 1$. When we use $\cos\alpha_o= \zeta /\ell$ in \eqref{polynomial-m=0} we obtain the equation
\[
2 \ell ^4 \,\left(\frac{\zeta ^4}{\ell ^4} (18-5 r)+\frac{\zeta ^2}{\ell ^2}  (5 r-1)+1\right)=0\,.
\]
Then a routine analysis shows that, for any $r \geq 2$, there exists no acceptable solution.
\vspace{2mm}

The other possible case occurs when $\beta(s)$ is a non constant solution of \eqref{eq:betaprimo-m=0}, i.e.,
\begin{equation}\label{beta}
\beta(s)=(\pm \zeta+\ell \cos\alpha_o)s+c
\end{equation}
where $c\in\R$ is an integration constant. In this case, we already know that proper biharmonic $3$-Frenet helices do exist (see \cite{Caddeo-altri-Mediterr-2006, MR2129448}). Here we extend this result to the case $r \geq 3$. More precisely, we prove the following existence result for proper $r$-harmonic helices into the Heisenberg space:
\begin{theorem}\label{Th-Heis-existence} Let $\ell \neq 0, r \geq 2$. Then there exists a proper $r$-harmonic $3$-Frenet helix $\gamma:\R \to M_{0,\ell}^3$. The Frenet field of $\gamma$ is as in \eqref{frenetN3=0}, where $\beta(s)$ is a non constant function of the type \eqref{beta}.
\end{theorem} 
\begin{proof} We assume $\ell>0$ (the case $\ell <0$ is analogous). We set $y=\zeta/\ell$ and rewrite the polynomial $P_4(\zeta)\lvert_{m=0}$ in \eqref{polynomial-m=0} as follows:
\[
\ell^4 \, P(y)\,,
\]
where
\begin{eqnarray*}\label{polynomial-m=0-bis}
P(y)&=&8 y^4+16  \cos\alpha_o y^3-\left[ (3 (r-2) \cos (2 \alpha_o)-3 r-4)\right]y^2\\\nonumber
    &&+2 \cos \alpha_o \left[1+2
   r \sin^2\alpha_o\right]y+2 \sin^2\alpha_o \,.
\end{eqnarray*}
Then, in order to complete the proof, it is sufficient to show that for all $r \geq 2$ there exists $0< \alpha_0 < \pi$ such that $P(y)$ admits a positive root.
To this purpose we observe that $\lim_{y \to +\infty}P(y)=+\infty$ and $P(0)>0$. Then it is enough to show that there exists $0< \alpha_0 < \pi$ such that $P(1/2)<0$ for all $r \geq 2$.  Now
\[
P(1/2)=\cos ^2\left(\frac{\alpha_0 }{2}\right) \big(-2 r \cos (2 \alpha_0 )+(r+2) \cos (\alpha_0 )+r+4 \big )\,.
\]
A direct inspection shows that, for any fixed $r\geq 2$, the equation 
\[
-2 r \cos (2 \alpha_0 )+(r+2) \cos (\alpha_0 )+r+4=0
\]
has a unique solution $\alpha^*\in(0,\pi)$ given by
\[
\cos (\alpha^* )=-\frac{\sqrt{49 r^2+68 r+4}-r-2}{8 r}\in\left(-1,-\frac{3}{4}\right)\,.
\]
Finally, a routine analysis now shows that, if  
$\alpha^*<\alpha_0<\pi$, then we have $P(1/2)<0$ as required to end the proof.

\end{proof}
\begin{remark} 
The explicit parametrization of a curve $\gamma$ as in Theorem\link\ref{Th-Heis-existence} can be easily obtained by direct integration of $T(s)$ (see \cite[Theorem 3]{MR2129448}).
\end{remark}
%

\subsection{The case $m\neq 0$}
When $m\neq 0$, the parametrizations of $3$-Frenet helices in $M_{m,\ell}^3$  satisfying $N_3=0$ was obtained in \cite[Theorem~5.6]{Caddeo-altri-Mediterr-2006}. More precisely, they proved:

\begin{theorem}\cite{Caddeo-altri-Mediterr-2006}\label{2006} Suppose that $\gamma(s)$ is an arc-length parametrized $3$-Frenet helix in a BCV-space $M_{m,\ell}^3$ with $4m\neq \ell^2$ and $m\neq 0$ satisfying $N_3=0$. Then the parametrization of $\gamma(s)$ is of one of the following three types:
\begin{enumerate}
\item[Type I]
\[
\begin{cases}
x(s)=\mu\sin\alpha_o\sin\beta(s)+c_1\,,\\
y(s)=-\mu\sin\alpha_o\cos\beta(s)+c_2\,,\\
z(s)=\frac{\ell}{4m}\beta(s)+\frac{1}{4m}\left(\left[4m-\ell^2\right]\cos\alpha_o-\ell\zeta\right)s\,,
\end{cases}
\]
where $\mu>0$ and $c_1$, $c_2$ are constants satisfying
\[
c_1^2+c_2^2=\frac{\mu}{m}\left(\left[\ell\cos\alpha_o+\zeta-\frac{1}{\mu}\right]+m\mu\sin^2\alpha_o\right)
\]
and $\beta(s)$ is  a nonconstant solution of
\[
\beta'=c_1 \sin\beta-c_2 \cos\beta+2m\mu\sin^2\alpha_o+\zeta+\ell \cos\alpha_o\,.
\]
\item[Type II] $\,$

\noindent If $\beta(s)=\beta_o$ is a constant such that $\sin\beta_o\cos\beta_o\neq 0$,
\[
\begin{cases}
x(s)=x(s)\,,\\
y(s)=x(s)\tan\beta_o+c_1\,,\\
z(s)=\frac{1}{4m}\left(\left[4m-\ell^2\right]\cos\alpha_o-\ell\zeta\right)s+c_2\,,
\end{cases}
\]
where $c_2\in\mathbb{R}$, the constant $c_1$ is given by
$$c_1=\frac{\zeta+\ell \cos\alpha_o}{2m\sin\alpha_o\cos\beta_o}$$
and $x(s)$ is a solution of the ordinary differential equation
$$x'(s)=\left(1+m\left[x^2(s)+\left(x(s)\tan\beta_o+c_1\right)^2\right]\right)\sin\alpha_o\cos\beta_o\,.$$
\item[Type III]$\,$

\noindent Let $\beta(s)=\beta_o$ be a constant satisfying $\sin\beta_o\cos\beta_o=0$. Here we describe the case that $\beta_0=\pi/2$, the case $\beta_0=0$ is similar up to interchanging the roles of $x$ and $y$:
\[
\begin{cases}
x(s)=x_o=- \frac{\zeta+\ell\cos\alpha_o}{2m\sin\alpha_o}\,,\\
y(s)=y(s)\,,\\
z(s)=\frac{1}{4m}\left(\left[4m-\ell^2\right]\cos\alpha_o-\ell\zeta\right)s+c_1\,
\end{cases}
\]
\noindent for a constant $c_1\in\mathbb{R}$ and where $y(s)$ is a solution of the ordinary differential equation
\begin{equation}\label{ODE-TypeIII}
\left(y'(s)\right)^2=\left(1+m\left[x_o^2+y^2(s)\right]\right)^2\sin^2\alpha_o\,.
\end{equation}
\end{enumerate}
\end{theorem}
Now we are in the right position to summarize our analysis in the following result:
\begin{theorem}\label{others} Let $r\geq 2$, $4m\neq \ell^2$, $m\neq 0$ and assume that $\zeta$ is a positive root of the polynomial \eqref{Pzeta}. Then
\begin{itemize}
\item[{\rm (i)}] any $3$-Frenet helix as in Theorem\link\ref{2006} is proper $r$-harmonic in $M_{m,\ell}^3$;
\item[{\rm (ii)}] conversely, any proper $r$-harmonic $3$-Frenet helix in $M_{m,\ell}^3$ is of this form.
\end{itemize}
\end{theorem}
\begin{proof} The first statement is a consequence of Proposition\link\ref{Prop-taur-BCV} and the subsequent above derivation of the polynomial $P_4(\zeta)$ in \eqref{Pzeta}. 

The converse follows readily from Lemma\link\ref{Th-N3=0-BCV} and Theorem\link\ref{2006}.
\end{proof}
We point out that, despite Theorem\link\ref{others}, an explicit description of all the $r$-harmonic $3$-Frenet helices in the case $4m\neq \ell^2$ and $m\neq 0$ is not an easy task. The difficulty is to understand when, given $m \neq 0, \ell$ and $ r \geq 2 $, there exists an angle $0 < \alpha_0 < \pi$ such that the polynomial $P_4(\zeta)$ defined in \eqref{Pzeta} admits a positive root.  

A simple case occurs when the BCV-space is a product. More precisely, in the case $\ell=0$ the polynomial \eqref{Pzeta} becomes:
$$
P_4(\zeta)\lvert_{\ell=0}=8 \zeta ^2 \big[\zeta ^2-4 \ m (r-1) \sin ^2 \alpha_0 \big ]\,.
$$
Then it is immediate to deduce the following corollary:

\begin{corollary}\label{Cor-H2xR-TypeIII}
{\rm (i)} Let $m<0$, $r \geq 2$. Then there exists no proper $r$-harmonic $3$-Frenet helix in $\h^2(4m) \times \R$.

{\rm (ii)} Let $m>0$, $r \geq 2$. Then there exist proper $r$-harmonic $3$-Frenet helices of Types I, II and III in $\s^2(\frac{1}{2\sqrt{m}}) \times \R$.
\end{corollary}

\begin{remark} The cases $r=2,3$ of Theorem\link\ref{others} were obtained in \cite{Caddeo-altri-Mediterr-2006, MR4308322} respectively. 

We also point out here that it was proved in \cite{Caddeo-altri-Mediterr-2006} that any biharmonic curve into the product space $\h^2(4m) \times \R$ is a geodesic, and it was deduced in \cite{MR4308322} that there exists no proper triharmonic helix in $\h^2(4m) \times \R$.
\end{remark}
Another consequence of Theorem\link\ref{others} is the following:
\begin{corollary}\label{Cor-BCV-general} Assume that $\ell \neq 0$ and $4m - \ell^2 >0$. Let $r\geq 2$. Then there exist proper $r$-harmonic $3$-Frenet helices in $M_{m,\ell}^3$ of Types I, II, III as in Theorem\link\ref{2006}.
\end{corollary}
\begin{proof}
It suffices to observe that under the assumptions of this corollary the polynomial $P_4(\zeta)$ in \eqref{Pzeta} satisfies $P_4(0)<0$ and $\lim_{\zeta \to + \infty}P_4(\zeta)=+\infty$. Then $P_4(\zeta)$ admits a positive root in all these cases.
\end{proof}
\begin{remark} We carried out some numerical studies which showed the existence of proper $r$-harmonic $3$-Frenet helices in $M_{m,\ell}^3$ also in the case that $4m-\ell^2<0$. For instance, using an argument similar to the proof of Theorem\link\ref{Th-Heis-existence}, we found existence when $m=1/4,\,\ell=2$, and also when $m=-1,\,\ell=1$. Since these arguments have a rather technical nature and do not open any new interesting path we omit further details.
\end{remark}
\section{$r$-harmonic full $n$-Frenet helices in $\s^m$}\label{n-FrenethelicesinSn}
 
In Theorem~\ref{Th-non-existence-r-curves} we obtained a non-existence result for full, $r$-harmonic $n$-Frenet curves in $\s^m$, $m \geq n$, under the assumption $n \geq 2r$. 

In this section we shall show that the hypothesis $n \geq 2r$ can be weakened if we assume that the curve is a helix. In this case, the proof of non-existence becomes equivalent to the fact that a system of polynomial equations admits no acceptable solution. To achieve this, in some cases we shall use the notion of a Groebner basis.

In order to describe in detail the general situation, let us first recall that the cases $n=3,\,r \geq 2$ have been fully understood in Theorem 1.1 and Corollary 1.4 of \cite{MR4542687}. Also the case $n=2$ is completely understood. Indeed, the full, $r$-harmonic $2$-Frenet helices in $\s^m$ are determined by the condition $\kappa^2=r-1$ and are the well-known circles of radius $R=1/\sqrt {r}$ (see \cite{MR3711937} and references therein).

By way of summary, the open cases are $n\geq 4, \,r \geq 3,\,n < 2r$. 

Our first result in this context is non-existence for the triharmonic case. Indeed,
\begin{theorem}\label{Th-Non-exist-3-harmonic-full-Frenet-helices}
Let $\gamma$ be a full $n$-Frenet helix in $\s^m$, $m\geq n \geq 4$. Then $\gamma$ cannot be triharmonic.
\end{theorem}
\begin{proof}
First, we observe that according to the previous discussion the open cases are $n=4$ and $n=5$.

Now, a simple computation shows that, for all $\ell \geq 0$, we can write
\begin{equation}\label{b[l,j]}
\nabla_T^\ell T=\sum_{j=1}^n b[\ell,j] F_j
\end{equation}
where the coefficients $b[\ell,j]$ are defined recursively as follows:
\begin{equation*}
b[0,j]=\delta_{1j}\,; \quad b[\ell,j]=\kappa_{j-1}\, b[\ell-1,j-1]-\kappa_j \,b[\ell-1,j+1] 
\end{equation*}
with boundary conditions $\kappa_0=\kappa_{-1}=\kappa_n=0$. Using this expression of $\nabla_T^\ell T$ and \eqref{Curv-tensor-sfera} we can easily compute the tension field $\tau_r(\gamma)$ as prescribed by \eqref{r-harmonicity-curves}. 

We study the two cases separately.

\textbf{Case $n=4$:}

The components of $\tau_3(\gamma)$ with respect to the base  $\{F_1,\ldots,F_4 \}$ are
\[
 \left(0,\kappa_1 \Big [\kappa_1^4+2 \kappa_1^2 (-1+\kappa_2^2)+\kappa_2^2 (-1+\kappa_2^2+\kappa_3^2)\Big],0,-\kappa_1 \kappa_2 \kappa_3 (-1+\kappa_1^2+\kappa_2^2+\kappa_3^2)\right )\,.
\]
Thus the condition of triharmonicity is equivalent to the following polynomial system:
\begin{equation}\label{3-harm-4-Frenet-helix}
\left \{ 
\begin{array}{ll}
{\rm (i)}&\kappa_1^4+2 \kappa_1^2 (-1+\kappa_2^2)+\kappa_2^2 (-1+\kappa_2^2+\kappa_3^2)=0\\
&\\
{\rm (ii)}&\kappa_1^2+\kappa_2^2+\kappa_3^2=1 \,.
\end{array}
\right .
\end{equation}
Now, using \eqref{3-harm-4-Frenet-helix}(ii) into \eqref{3-harm-4-Frenet-helix}(i) and simplifying we obtain:
\[
\left (1+\kappa_3^2 \right ) \left (-1+\kappa_2^2+\kappa_3^2 \right )=0 \,.
\]
Since $\kappa_1>0$, the last equation is not compatible with \eqref{3-harm-4-Frenet-helix}(ii) and so this case is ended.
\textbf{Case $n=5$:}

In this case the condition of triharmonicity is equivalent to the following polynomial system:
\begin{equation}\label{3-harm-5-Frenet-helix}
\left \{ 
\begin{array}{ll}
{\rm (i)}&\kappa_1^4+2 \kappa_1^2 (-1+\kappa_2^2)+\kappa_2^2 (-1+\kappa_2^2+\kappa_3^2)=0\\
&\\
{\rm (ii)}&\kappa_1^2+\kappa_2^2+\kappa_3^2+\kappa_4^2=1 \,.
\end{array}
\right .
\end{equation}
Next, to simplify the notation, we write
\[
x=\kappa_1^2 \,, \quad y=\kappa_2^2 \,, \quad z=\kappa_3^2 \,, \quad t=\kappa_4^2 \,.
\]
Then our system \eqref{3-harm-5-Frenet-helix} becomes:
\begin{equation}\label{3-harm-5-Frenet-helix-bis}
\left \{
\begin{array}{ll}
{\rm (i)}&x^2+2 x (-1+y)+y (-1+y+z)=0\\
&\\
{\rm (ii)}&x+y+z+t=1 \,.
\end{array}
\right .
\end{equation}
Now, from \eqref{3-harm-5-Frenet-helix-bis}(i) we deduce
\begin{equation}\label{z}
z=\frac{-x^2-2 x y+2 x-y^2+y}{y}\,.
\end{equation}
Using this condition in \eqref{3-harm-5-Frenet-helix-bis}(ii) we find
\[
-t y+x^2+x (y-2)=0\,.
\]
Now, if $t=x$, we have an immediate contradiction. If $t\neq x$, then
\begin{equation}\label{y}
y=\frac{x^2-2 x}{t-x}
\end{equation}
which, from \eqref{z}, gives
\begin{equation}\label{z-bis}
z=\frac{-t^2+t+x}{t-x}\,.
\end{equation}
Finally, we show that \eqref{y} and \eqref{z-bis} lead us to a contradiction. Indeed, $y>0$ implies $t<x$ from \eqref{y}. But $t<x$ and \eqref{z-bis} easily imply $z<0$, a contradiction and so the proof is completed.
\end{proof}
\vspace{3mm}

Using the explicit expression of $\tau_r{\gamma}$ which can always be computed using \eqref{b[l,j]} and \eqref{Curv-tensor-sfera} it is possible to investigate any instance which falls into the range of the remaining cases of interest, i.e., $n\geq 4, \,r \geq 4,\,n < 2r$.

A natural next step forward would be to attack the case of full, $r$-harmonic $4$-Frenet helices when $r\geq 4$. 

Here the computational complexity rapidly increases with $r$ and we propose a method of proof which uses the notion of a Groebner basis (see \cite{MR1213453}) and enables us to solve the cases where $r=4,5$.
More precisely, we obtain the following non-existence result:
\begin{theorem}\label{Th-Non-exist-4-harmonic-full-Frenet-helices}
Let $\gamma$ be a full $4$-Frenet helix in $\s^m$, $m\geq 4$. Then $\gamma$ cannot be $r$-harmonic, $r=4,5$.
\end{theorem}
\begin{proof} We examine the two cases separately.
\vspace{2mm}

\textbf{Case $r=4$:}

A computation shows that the components of $\tau_4(\gamma)$ with respect to the Frenet field $\left \{F_1,\ldots,F_4 \right \}$ are 
\[
\Big(0,\kappa_1\,P_1(\kappa_1,\kappa_2,\kappa_3),
0,\kappa_1\kappa_2 \kappa_3 P_2(\kappa_1,\kappa_2,\kappa_3)\Big)\,.
\]
where
\[
\begin{split}
P_1(\kappa_1,\kappa_2,\kappa_3)=-\kappa_1^6-3 \left(\kappa_2^2-1\right) \kappa_1^4+\kappa_2^2 \left(-3 \kappa_2^2-2 \kappa_3^2+4\right) \kappa_1^2\\
+\kappa_2^2 \left(-\kappa_2^4+\left(1-2 \kappa_3^2\right) \kappa_2^2-\kappa_3^4+\kappa_3^2\right)
\end{split}
\]\,
and
\[
P_2(\kappa_1,\kappa_2,\kappa_3)=\left(\kappa_1^4+\left(2 \kappa_2^2+\kappa_3^2-2\right) \kappa_1^2+\kappa_2^4+\kappa_3^2 \left(\kappa_3^2-1\right)+\kappa_2^2 \left(2 \kappa_3^2-1\right)\right)\,.
\]
We have to show that the polynomial system of equations
\[
P_1(\kappa_1,\kappa_2,\kappa_3)=0,\quad P_2(\kappa_1,\kappa_2,\kappa_3)=0
\]
admits no solution with $\kappa_i>0, \,i=1,2,3$. 

It is convenient to introduce an auxiliary polynomial 
\[
P_3(\kappa_1,\kappa_2,\kappa_3)= \kappa_1^2+\kappa_2^2+\kappa_3^2 -c\,,
\]
where $c>0$ is a real constant. Now, the idea of the proof is the following (see \cite{MR1213453}). We consider the ideal $\mathcal{I}=\langle P_1,P_2,P_3 \rangle$ of $\C[\kappa_1,\kappa_2,\kappa_3]$ generated by the polynomials $P_i,\,i=1,2,3$. The associated algebraic variety can be described by means of a Groebner basis of the ideal $\mathcal{I}$. It follows that the solutions of the polynomial system
\[
P_1(\kappa_1,\kappa_2,\kappa_3)=0,\quad P_2(\kappa_1,\kappa_2,\kappa_3)=0,\quad P_3(\kappa_1,\kappa_2,\kappa_3)=0
\] 
correspond to the common zeroes of the elements of a Groebner basis of the ideal $\mathcal{I}$. Next, the Mathematica$^{\tiny{\textcircled{R}}}$ code
\[
{\rm GroebnerBasis}[\{P_1[\kappa_1,\kappa_2,\kappa_3],P_2[\kappa_1,\kappa_2,\kappa_3],P_3[\kappa_1,\kappa_2,\kappa_3]\},\{\kappa_1,\kappa_2,\kappa_3 \}]//{\rm FullSimplify}
\]
gives a convenient, simplified Groebner basis $\mathcal{G}\mathcal{B}$ of $\mathcal{I}$. The outcome is:
\[
\mathcal{G}\mathcal{B}=\left\{(c-1) c \left((c-2) \kappa_3^2+c\right),2 \left(\kappa_2^2+\kappa_3^2\right)-c ((c-3) c+4),(c-2) (c-1) c+2 \kappa_1^2\right\}\,.
\]
By way of summary, it remains to prove that, for all $c>0$, the system
\begin{equation}\label{SystemGBr=4n=4}
\left \{
\begin{array}{ll}
{\rm (i)} &c(c-1)  \left((c-2) \kappa_3^2+c\right)=0\\
{\rm (ii)} &2 \left(\kappa_2^2+\kappa_3^2\right)-c ((c-3) c+4)=0\\
{\rm (iii)} &c(c-2) (c-1)+2 \kappa_1^2=0
\end{array}
\right .
\end{equation}
admits no solution with $\kappa_i>0, \,i=1,2,3$. From equation \eqref{SystemGBr=4n=4}(iii) we deduce that necessarily $1<c<2$. Next, from equation \eqref{SystemGBr=4n=4}(i) we infer
\[
\kappa_3^2=\frac{c}{(2-c)}\,.
\]
Substituting this condition into \eqref{SystemGBr=4n=4}(ii) we find:
\[
\frac{c (c-1) ((c-4) c+6)}{(2-c)}+2 \kappa_2^2 =0\,.
\]
This is a contradiction since it is easy to check that
\[
\frac{c(c-1) ((c-4) c+6)}{(2-c)}
\]
is positive when $1<c<2$, so ending the proof of the case $r=4$.
\vspace{2mm}

\textbf{Case $r=5$:}

The scheme of the proof is identical to that of the previous case. Now we have:
\[
\begin{aligned}
P_1(\kappa_1,\kappa_2,\kappa_3)=&\kappa_1^8+4 \left(\kappa_2^2-1\right) \kappa_1^6+\kappa_2^2 \left(\kappa_2^2+\kappa_3^2-1\right) \left(\kappa_2^2+\kappa_3^2\right)^2\\
&+3 \kappa_2^2 \left(2 \kappa_2^2+\kappa_3^2-3\right) \kappa_1^4+2 \left(2 \kappa_2^6+3 \left(\kappa_3^2-1\right) \kappa_2^4+\kappa_3^2 \left(\kappa_3^2-2\right) \kappa_2^2\right) \kappa_1^2
\end{aligned}
\]\,
and
\[
\begin{aligned}
P_2(\kappa_1,\kappa_2,\kappa_3)=&\kappa_1^6+\left(3 \kappa_2^2+\kappa_3^2-3\right) \kappa_1^4+\left(\kappa_2^2+\kappa_3^2-1\right) \left(\kappa_2^2+\kappa_3^2\right)^2\\
&+\kappa_1^2 \left(3 \kappa_2^4+4 \left(\kappa_3^2-1\right) \kappa_2^2+\kappa_3^2 \left(\kappa_3^2-2\right)\right)\,.
\end{aligned}
\]
Next, after computation of a Groebner basis, we find that system \eqref{SystemGBr=4n=4} here takes the following form:
{
\begin{equation}\label{SystemGBr=5n=4}
\begin{cases}
{\rm (i)}\quad (c-1)  \left(2 c^2+(c-4) (c-1) \kappa_3^4+2 (c-2) c \kappa_3^2\right)=0\\
{\rm (ii)}\quad 2 (c-2) \kappa_2^2-c (c(c-6) +7) \kappa_3^2-c \left(c^2+c-4\right)=0\\
{\rm (iii)}\quad c \left(2 ((c-1) c-1) \kappa_2^2-c (c+1) \left(c^2+c-3\right)\right)+\\
\qquad -\kappa_3^2 \left(c^5-5 c^4+2 c^3+2 c^2+4 c-2 \kappa_2^2\right)+2 \kappa_3^4=0\\
{\rm (iv)}\quad 4 c^2 + 9 c^3 - 10 c^4 + c^5 + 8 c \kappa_2^2 - 8 c^2 \kappa_2^2 + \\
 \qquad -  8 \kappa_2^4 + 4 c \kappa_2^4 - (4 c^4-20 c^3+16 c^2+8 c) \kappa_3^2 - 
   4 (-2 + c) \kappa_3^4=0\\
{\rm (v)}\quad \kappa_1^2 + \kappa_2^2 + \kappa_3^2-c=0\,.
\end{cases}
\end{equation}
}

Then it remains to show that, for all $c>0$, this system admits no solution with $\kappa_i>0, \,i=1,2,3$. 

First, it is easy to check that there is no admissible solution if $c=1$. Therefore, from now on, we assume that $c \neq 1$. Then, from equation \eqref{SystemGBr=5n=4}(i) we find
\begin{equation}\label{kappa3-caso1}
\kappa_3^2=\frac{2 c-c^2}{(c-1)(c-4)}-\sqrt{\frac{-c^4+6 c^3-4 c^2}{(c-1)^2(c-4)^2}} \quad {\rm if}\,\, 3-\sqrt{5}<c<1 \,,
\end{equation}
or
\begin{equation}\label{kappa3-caso2}
\kappa_3^2=\frac{2 c-c^2}{(c-1)(c-4)}+\sqrt{\frac{-c^4+6 c^3-4 c^2}{(c-1)^2(c-4)^2}} \quad {\rm if}\,\, 3-\sqrt{5}<c<1 \,\,{\rm or}\,\,1<c<4\,.
\end{equation}
First, we analyse case \eqref{kappa3-caso1}. Replacing this value of $\kappa_3^2$ into \eqref{SystemGBr=5n=4}(ii) we deduce

\begin{equation}\label{kappa2}
\kappa_2^2=\frac{c(c-4) (c-1) \left(c^2+c-4\right)+ ((c-6) c^3+7c^2) \left(2-c-\sqrt{-c^2+6 c-4}\right)}{2 (c-1) (c-2) (c-4)}\,.
\end{equation} 

Next, inserting \eqref{kappa2} and \eqref{kappa3-caso1} into \eqref{SystemGBr=5n=4}(v) we obtain:
\[
g(c)+\kappa_1^2=0\,,
\]
where
\[
g(c)=\frac{c(c-1)  \left(2-\sqrt{-c^2+6 c-4}\right)}{2 (c-2)}\,.
\]
Finally, a straightforward analysis shows that the function $g(c)$ is nonnegative on the interval $3-\sqrt{5}<c<1$ and this ends the proof in the case \eqref{kappa3-caso1}.

The case \eqref{kappa3-caso2} is similar and so we omit the details.
\end{proof}
\section{Appendix}\label{appendix}
In this paper we focused on the study of $r$-harmonic \textit{helices}. A natural, related general question is:

\begin{problem} Let $\gamma$ be an arc length parametrized proper $r$-harmonic curve in a Riemannian manifold $(M,g)$. Under which assumptions is $\gamma$ necessarily a helix?
\end{problem}

We know that in general the answer is negative, as shown in Proposition~3.3 of \cite{MR4308322}, where the authors proved the existence of a triharmonic curve with non constant curvature and torsion in a surface. If the dimension of the ambient space $M$ is greater than $2$ also partial answers would be interesting and the problem is open even in the case that $M=\s^m$.

In this appendix we discuss briefly some related results.

First, we recall the following fact which derives from the explicit analysis of the tangential component of the $3$-tension field $\tau_3(\gamma$):

\begin{proposition} \cite{MR4308322} Let $\gamma$ be an arc length parametrized proper triharmonic curve in a $3$-dimensional Riemannian manifold $(M^3,g)$. Then
\begin{equation}\label{integraleprimo}
2 \frac{d}{ds} \Big (k(s)^2 k''(s) \Big )= k(s)\frac{d}{ds} \Big (k(s)^2 \left [ k(s)^2+\tau(s)^2\right ] \Big )\,.
\end{equation}
\end{proposition}
As an immediate corollary, we conclude that \textit{if $k(s)$ is a constant, then also $\tau(s)\equiv \tau_o \in \R$. }

As a partial converse, in \cite{MR4308322} the authors obtained: 
\begin{proposition}\label{prop4-mont-Pampano} \cite{MR4308322} Let $\gamma$ be an arc length parametrized proper triharmonic curve in a $3$-dimensional BCV-space $M^3_{m,\ell}$, $4m \neq \ell^2$. If $\tau(s)\equiv 0$, then $k(s)$ is constant.
\end{proposition}
The proof of this proposition was only rapidly sketched in \cite{MR4308322} and some conclusions, although correct, do not follow directly from the arguments indicated.  Therefore, we take this opportunity to insert here the missing details. A further motivation to write down this proof lies in the fact that an adaptation of this method will enable us to obtain a stronger conclusion in the case that the ambient space is $\s^m$. More precisely, we shall prove:
\begin{proposition}\label{prop-in-sm} Let $\gamma$ be an arc length parametrized proper triharmonic $3$-Frenet curve in $\s^m$, $m\geq 3$. If $\tau(s)\equiv \tau_o \in \R$, then $k(s)$ is constant.
\end{proposition}
\begin{proof}[Proof of Proposition~\ref{prop4-mont-Pampano}] By assumption, $\tau(s)\equiv 0$.  The condition $\tau_3(\gamma)=0$, where  $\tau_3(\gamma)$ is given in \eqref{r-harmonicity-curves}, can be expressed by the vanishing of the components of  $\tau_3(\gamma)$ with respect to the base $\{T,N,B \}$. By a straightforward computation, taking into account \eqref{curv-3}, we thus obtain that $\gamma$ is triharmonic if and only if

\begin{equation}\label{eq:apx1}
\begin{cases}
-k^{(3)} k+2 k^3 k'-2 k' k''=0\\
4k^{(4)}+ \left(16 m B_3^2-\ell ^2 \left(4 B_3^2-1\right)\right)(k''-2 k^3)-40 k^2 k''-60 k k'^2+4k^5=0\\
N_3 B_3 \left(4 m-\ell ^2\right) \left(2 k^3-k''\right)=0\,.
\end{cases}
\end{equation}
Now, let us first examine the component of $\tau_3(\gamma)$ along $B$, that is the third equation of \eqref{eq:apx1}. We assume that $N_3B_3 \neq 0$ near some point $s_0$. Then, in this open set we must have
\[
k''=2 k^3\,.
\]
Inserting this condition into the first equation of \eqref{eq:apx1} it is easy to deduce that
\[
-8 k^3 k'=0
\]
and so $k(s)$ must be a constant. Therefore, it remains to analyse the case that $N_3 B_3=0$ along $\gamma$. First, let us assume that $N_3=0$. Then it follows from \eqref{BprimeTprime} that $B_3$ is constant. Next, using the prime integral \eqref{integraleprimo}, performing a further integration and the necessary replacements, it is easy to deduce:
\begin{equation}\label{derivate}
\begin{aligned}
k' =& \;\sqrt{\frac{k ^4}{5}-\frac{2 c_1} {k }+c_2}\\
k'' =&\; \frac{2 k ^3}{5}+\frac{c_1}{k ^2}\\
k^{(3)} =& \;\frac{6}{5} k'  k ^2-\frac{2 c_1 k' }{k ^3}=\frac{2 \left(3 k ^5-5 c-1\right) \sqrt{\frac{k ^4}{5}-\frac{2 c_1} {k }+c_2}}{5 k ^3}\\
k^{(4)} =& \;\frac{6 c_1{k'} ^2}{k ^4}-\frac{2 c_1k'' }{k ^3}+\frac{12}{5}{k'} ^2 k +\frac{6}{5} k''  k ^2=\frac{2 \big(2 k ^5+5 c_1\big) \big(6 k ^5+15 c_2 k -35 c_1\big)}{25 k ^5}
\end{aligned}
\end{equation}
where $c_1, c_2$ are integration constants. To be precise, we point out that also the case
\[
k'=-\sqrt{\frac{k ^4}{5}-\frac{2 c_1} {k }+c_2}
\]
could occur, but all what follows would work analogously and so we just describe the case \eqref{derivate}. Finally, we insert all the expressions \eqref{derivate} in the normal component of $\tau_3(\gamma)$, that is the second equation of \eqref{eq:apx1}. Performing the necessary simplifications we deduce that

\[
\frac{126}{25}  k ^{10}- \frac{1}{20}\left(4 B_3^2 \big(\ell ^2+4 m\right)-\ell ^2\big)(8k^8-5 c_1 k ^3)+\frac{63}{5} c_2 k ^6-\frac{84}{5} c_1 k ^5-6c_1c_2 k+14 c_1^2=0\,.
\]
It follows that $k(s)$ is a root of a not identically zero polynomial with constant coefficients and so $k(s)$ is constant, as required. If $B_3=0$, then $B_3$ is a constant and so the end of the proof is the same. 
\end{proof}
\begin{proof}[Proof of Proposition~\ref{prop-in-sm}] By assumption, $\tau(s)\equiv \tau_o$. Then a straightforward computation shows that the vanishing of the components of $\tau_3(\gamma)$ with respect to the base $\{T,N,B \}$ is equivalent to the following system:
\begin{equation}\label{eq:apx2}
\begin{cases}
2 k ^3 k' +k  \left(\tau_o ^2 k' -k^{(3)} \right)-2 k'  k'' =0\\
k^{(4)} +\left(1-6 \tau_o ^2\right) k'' -10 k ^2 k''+k  \left(-15 {k'} ^2+\tau_o ^4-\tau_o ^2\right)+2 \left(\tau_o ^2-1\right) k ^3+k ^5=0\\
4\tau_o  k^{(3)} + \tau_o \big(-9 k ^2-4 \tau_o ^2+2\big) k' =0\,.
\end{cases}
\end{equation}
Next, we use again the prime integral \eqref{integraleprimo} and derive the analogous of \eqref{derivate} in this context:
\begin{equation}
\begin{aligned}\label{derivate2}
k' =&\; \sqrt{\frac{k ^4}{5}+\frac{1}{3} \tau_o ^2 k ^2-\frac{2 c_1}{k }+c_2}\\
k'' =&\; \frac{c_1}{k ^2}+\frac{1}{3} \tau_o ^2 k +\frac{2 k ^3}{5}\\
k^{(3)} =&\; \frac{\big(18 k ^5+5 \tau_o ^2 k ^3-30 c_1\big) \sqrt{\frac{k ^4}{5}+\frac{1}{3} \tau_o ^2 k ^2-\frac{2 c_1}{k }+c_2}}{15 k ^3}\\
k^{(4)} =&\;\frac{24 k ^5}{25}+\frac{4}{3} \tau_o ^2 k ^3 +\left(\frac{12 c_2}{5}+\frac{\tau_o ^4}{9}\right) k -\frac{14 c_1^2}{k ^5}+\frac{6 c_1c_2}{k ^4}+\frac{5 c_1 \tau_o ^2}{3 k ^2}-\frac{16 c_1}{5} 
\end{aligned}
\end{equation}
where $c_1, c_2$ are integration constants. Finally, we insert all the expressions \eqref{derivate2} in the second equation of \eqref{eq:apx2}. Performing the necessary simplifications we deduce that
{\small
\[
\begin{aligned}
\frac{126}{25}  k ^{10}+\frac{37 \tau_o ^2+8}{5} k ^8+\frac{567 c_2+40 \tau_o ^4+30 \tau_o ^2}{45} k ^6
-\frac{84}{5} c_1 k ^5-\frac{3c_1-13 c_1 \tau_o ^2}{3} k ^3-6 c_1 c_2 k +14 c_1^2=0\,.
\end{aligned}
\]
}

\noindent Therefore, $k(s)$ is a root of a not identically zero polynomial with constant coefficients and so $k(s)$ is constant, as required to end the proof.
\end{proof}

%
%

%

{\bf Funding} The authors S.M. and A.R. are members of the Italian National Group G.N.S.A.G.A. of INdAM. The work was partially supported by the Project {ISI-HOMOS} funded by Fondazione di Sardegna and by PNRR e.INS Ecosystem of Innovation for Next Generation Sardinia (CUP F53C22000430001, codice MUR ECS00000038)

\section*{Declarations}

{\bf Conflict of interest} The authors declare that they have no competing interests. 

{\bf Ethical approval} Not applicable.

{\bf Data availability} Not applicable.

\end{document}